\documentclass[12pt]{article}
\usepackage{amssymb,amsmath,latexsym}
\usepackage{amsthm}
\newtheorem{theorem}{Theorem}

\begin{document}
\author{Alexander E Patkowski}
\title{A Strange Partition Theorem Related to the Second \newline Atkin-Garvan Moment}
\date{\vspace{-5ex}}
\maketitle
\abstract{This paper contain results on a strange smallest parts function related to the second Atkin-Garvan moment. Some
new identities are discovered in relation to Andrews $spt$ function as well as one of Borweins' two-dimensional theta functions.}

\section{Introduction}
The generating function for the sum of smallest parts among partitions of $n$ into an odd number of distinct parts minus the sum of the smallest parts among the partitions of $n$ into an even number of distinct parts is
$$\sum_{n\ge0}(1-(q^{n+1})_{\infty}),$$ where we use standard notation [8] $(a)_n=(a;q)_{n}:=\prod_{0\le k\le n-1}(1-aq^{k}),$ $q\in\mathbb{C}.$ In fact, it is known that (see [2])
\begin{equation}\sum_{n\ge0}(1-(q^{n+1})_{\infty})=\sum_{n\ge1}\sigma_{0}(n)q^n,\end{equation}
where $\sigma_{k}(n)=\sum_{d|n}d^k,$ $n\in\mathbb{N},$ and $\sigma_{k}(n)=0$ if $n\notin\mathbb{N}.$ (See [11] for similar identities.)
\par Andrews' [2] discovered the generating function for $spt(n),$ the total number of appearances of the smallest part in unrestricted partitions of $n,$
and then related it to the second Atkin-Garvan Moment. Namely, he proved [2]
\begin{equation}\sum_{n\ge1}spt(n)q^n=\sum_{n\ge1}\frac{q^n}{(1-q^n)(q^n)_{\infty}}=\sum_{n\ge1}np(n)q^n-\frac{1}{2}\sum_{n\ge1}N_2(n)q^n,\end{equation}
where $p(n)$ is the number of unrestricted partitions of $n,$ $N_2(n)=\sum_{m\in\mathbb{Z}}m^2N(m,n)$ is the second Atkin-Garvan moment, and $N(m,n)$ is the number of partitions of $n$ with rank $m.$ The ``rank" of a partition is the largest part minus the number of parts. The implication of relating $spt(n)$ to the right side of equation (2) is that $spt(n)$ satisfies some interesting Ramanujan-type congruences [1, 2].
\par The purpose of this paper is to offer a strange result similar to (2), and more than the contents of my preprint referenced in Garvan's paper [8]. \begin{theorem} We have,
$$\sum_{n\ge1}\frac{q^n}{(1-q^n)(1-q^n)(1-q^{n+1})\cdots(1-q^{2n-1})(q^{3n};q^3)_{\infty}}$$
\begin{equation}=\frac{1}{(q^3;q^3)_{\infty}}\sum_{n\ge1}\frac{nq^n}{1-q^n}-\frac{1}{2}\sum_{n\ge1}N_2(n)q^{3n}.\end{equation}

\end{theorem}
To prove this theorem, we require the machinery of Bailey pairs and Bailey's lemma. Recall [4] that we define a pair of sequences $(\alpha_n,\beta_n)$ to be a Bailey pair with respect to $a$ if 
		\begin{equation}\beta_n=\sum_{r\ge0}^{n} \frac{\alpha_r}{(aq;q)_{n+r} (q;q)_{n-r}}.\end{equation}
The following is Bailey's lemma, which utilizes this definition of ($\alpha_n$, $\beta_n$) to produce new $q$-series identities. \newline
{\bf Lemma 1.2} \it If $(\alpha_n,\beta_n)$ form a Bailey pair with respect to $a$ then
\begin{equation}\sum_{n\ge0}^{\infty}(z)_n(y)_n(aq/zy)^n\beta_n=\frac{(aq/z)_{\infty}(aq/y)_{\infty}}{(aq)_{\infty}(aq/zy)_{\infty}}\sum_{n\ge0}^{\infty}\frac{(z)_n(y)_n(aq/zy)^n\alpha_n}{(aq/z)_n(aq/y)_n}.\end{equation} \rm
From Slater's list, we have the following Bailey pair $(\alpha_n,\beta_n)$ relative to $a=1$ [12, J(1)]
\begin {equation} \alpha_{3n\pm1}=0,\end{equation}
\begin {equation} \alpha_{3n}=(-1)^nq^{3n(3n-1)/2}(1+q^{3n}),\end{equation}
\begin{equation} \beta_n=\frac{(q^3;q^3)_{n-1}}{(q)_n(q)_{2n-1}}.\end{equation}
\rm
Differentiating (5) with respect to $z$ and setting $z=1,$ and then doing the same for $y$ yields the identity (after setting $a=1$),
\begin{equation}\sum_{n\ge1}(q;q)_{n-1}^2\beta_nq^n=\alpha_0\sum_{n\ge1}\frac{nq^n}{1-q^n}+\sum_{n\ge1}\frac{\alpha_nq^n}{(1-q^n)^2}.\end{equation}
Inserting the Bailey pair (6)-(8) into equation (9) and then multiplying through by $(q^3;q^3)_{\infty}^{-1}$ gives us Theorem 1. To see that
$$-\frac{1}{2}\sum_{n\ge1}N_2(n)q^{3n}=\sum_{n\ge1}\frac{\alpha_nq^n}{(1-q^n)^2},$$
note [2, eq.(3.4)]
$$-\frac{1}{2}\sum_{n\ge1}N_2(n)q^{n}=\sum_{n\ge1}\frac{(-1)^nq^{n(3n+1)/2}(1+q^n)}{(1-q^{n})^2}.$$
\par
The generating function on the left hand side of (3) is $spt_{2,3}(n),$ the total number of appearances of the smallest part in each integer partition of $n,$ where 
parts are $<$ twice the smallest or multiples of three $\ge$ thrice the smallest.
\section{A Relation to the Borwein theta function $a(q)$}
One of the theta functions introduced by the Borwein's [5, Chapter 4], [6] is the function
\begin{equation}a(q):=\sum_{n,m\in\mathbb{Z}}q^{n^2+nm+m^2}.\end{equation}
The function $a(q)$ has been studied in considerable detail, and has several direct relations to Jacobi's theta functions (see [6]). The Lambert series expansion for $a(q)$ is due to Lorenz, and can be found in [6, eq.(2.21)]
$$a(q)=1+6\sum_{n\ge1}\left(\frac{q^{3n+1}}{1-q^{3n+1}}-\frac{q^{3n+2}}{1-q^{3n+2}}\right).$$
In [3, pg. 460, Entry 3(i)], we find
\begin{equation}\sum_{n\ge1}\frac{nq^{n}}{1-q^n}-3\sum_{n\ge1}\frac{nq^{3n}}{1-q^{3n}}=\frac{a^2(q)-1}{12}.\end{equation}
We are now ready to prove the following result.
\begin{theorem} If $n\equiv\pm1\pmod{3},$ and 
$$\sum_{m\ge1}\xi(m)q^m:=\frac{1}{(q^3;q^3)_{\infty}}\left(\frac{a^2(q)-1}{12}\right),$$
then $spt_{2,3}(n)=\xi(n).$ Further, if $n\equiv0\pmod{3},$ then
$$spt_{2,3}(n)=3spt(n/3)+\xi(n)+N_2(n/3).$$

\end{theorem}

The key identity to prove this result is,
$$\sum_{n\ge1}spt_{2,3}(n)q^n-3\sum_{n\ge1}spt(n)q^{3n}$$
\begin{equation}=\frac{1}{(q^3;q^3)_{\infty}}\left(\frac{a^2(q)-1}{12}\right)-\frac{2}{(q^3;q^3)_{\infty}}\sum_{n\ge1}\frac{(-1)^n q^{3n(3n+1)/2}(1+q^{3n})}{(1-q^{3n})^2}.\end{equation}
Equation (12) is three multiplied by equation (2) (with $q$ replaced by $q^3$) subtracted from Theorem 2. Note the last series on the right side of (12) is 
$\sum_{n\ge1}N_2(n)q^{3n}$ [2. eq.(3.4)]. The first part of Theorem 2 is obtained from taking the coefficient of $q^{3n\pm1}$ of equation (12), and noting that the generating functions over $q^{3n}$ get omitted when doing this. The second part follows from taking the coefficient of $q^{3n}$ on both sides. 
\section{Another look at Theorem 2}
In a paper by Huard, Ou, Spearman, and Williams [10] on convolution sums involving divisor functions, we find the following nice theorem:
 \newline
{\bf [10, Theorem 13] \bf} \it The number of representations of a positive integer $n$ by the quaternary form $x^2 + xy + y^2 + u^2 + uv + v^2$ is $12\sigma(n)-36\sigma(n/3).$
\\*
\\*
\rm
This theorem is attributed in [10] to G. A. Lomadze. It is not difficult to see that [10, Theorem 13] is essentially equivalent to equation (11) upon taking the coefficient of $q^n.$ Naturally, from this observation we can restate Theorem 2 
using the cardinality of a quaternary form. \par Following their notation [10], put 
$$R(k)=\mathbf{card}\{(x, y, u, v)\in\mathbb{Z}^4 \hspace{1mm} | \hspace{1mm} k=x^2 + xy + y^2 + u^2 + uv + v^2\}.$$

Using the above result we can now prove the following.
\begin{theorem} Let $p_3(n)$ be the number of partitions of $n$ with parts congruent to $0\pmod{3}.$ Put 
$$P_3(n)=\sum_{k}R(k)p_3(n-k),$$
then $spt_{2,3}(3n)\equiv \frac{1}{12}P_3(3n)-\frac{1}{2}N_2(n)\pmod{3}.$ Further, if $n\equiv\pm1\pmod{3}$
then $spt_{2,3}(n)=\frac{1}{12}P_3(n).$ \end{theorem}

\begin{proof} Using [10, Theorem 13], it can be seen that 
\begin{align} \frac{1}{(q^3;q^3)_{\infty}}\sum_{n\ge1}\frac{nq^n}{1-q^n}\\
&=\frac{1}{(q^3;q^3)_{\infty}}\sum_{n\ge1}\sigma(n)q^n \\
&=\frac{1}{(q^3;q^3)_{\infty}}\sum_{n\ge1}\left(\frac{1}{12}R(n)+3\sigma(n/3)\right)q^n\\
&=\frac{1}{(q^3;q^3)_{\infty}}\left(\frac{1}{12}\sum_{n\ge1}R(n)q^n+3\sum_{n\ge1}\frac{nq^{3n}}{1-q^{3n}}\right)\\
&=\frac{1}{12}\sum_{n\ge1}P_3(n)q^n+3\sum_{n\ge1}np(n)q^{3n},
\end{align}
since $\sigma(n/3)=0$ unless $n\equiv0\pmod{3}$ by definition. Line (17) follows from Euer's well-known identity $np(n)=\sum_{k\le n}p(k)\sigma(n-k).$ This observation coupled with equation (3), and equating coefficients of $q^{3n}$ and then $q^{3n\pm1}$ gives the result. \end{proof}
We note that one might also write $$P_3(3n)=\sum_{k}R(3k)p(n-k),$$ upon noting that the coefficient of $q^{3n}$ of $\sum_{n\ge0}P_3(n)q^n$ is
the coefficient of $q^{3n}$ in $$\sum_{m,k\ge0}p(m)R(3k)q^{3(m+k)}.$$
\section{A Relation to Andrews' $spt$ function}

In this section we offer a nice consequence of Theorem 1 and some concluding remarks. 
\begin{theorem}We have, $spt_{2,3}(3n)\equiv spt(n)\pmod{3}.$\end{theorem}
\begin{proof} By Theorem 1, we have
\begin{equation}\sum_{n\ge1}spt_{2,3}(n)q^n=\frac{1}{(q^3;q^3)_{\infty}}\sum_{n\ge1}\sigma(n)q^n-\frac{1}{2}\sum_{n\ge1}N_2(n)q^{3n}.\end{equation}
Note that $\sigma(3n)=4\sigma(n)-3\sigma(n/3),$ and hence $\sigma(3n)\equiv4\sigma(n)\pmod{3},$ and therefore $\sigma(3n)\equiv\sigma(n)\pmod{3}$ (by the triviality $a\equiv 4b\pmod{3}$ iff $a\equiv b\pmod{3}$). With this in mind, we see that the coefficient of $q^{3n}$ in (18) is
\begin{equation}spt_{2,3}(3n)=\sum_{k}p(k)\sigma(3(n-k))-\frac{1}{2}N_2(n).\end{equation}
This follows from the observation that the coefficient of $q^{3n}$ in
$$\sum_{k\ge1}p(k)q^{3k}\sum_{m\ge1}\sigma(m)q^{m},$$
is the coefficient of $q^{3n}$ in
$$\sum_{m,k\ge1}p(k)\sigma(3m)q^{3(m+k)}.$$ Now using $\sigma(3n)\equiv\sigma(n)\pmod{3}$ with equation (19), we find
\begin{align} spt_{2,3}(3n)\\
&\equiv \sum_{k}p(k)\sigma(n-k)-\frac{1}{2}N_2(n)\pmod{3} \\
&\equiv np(n)-\frac{1}{2}N_2(n)\pmod{3}\\
&\equiv spt(n)\pmod{3}.
\end{align}
\end{proof}
For a complete multiplicative theory of Andrews' $spt$ function modulo $3$ see [7]. It would be nice to see a similar theory built for $spt_{2,3}(n)$ using 
Theorem 4. \par We note it is natural to consider re-studying Theorem 2 (or $\xi(n)$) further by considering further identities for the $q$-series
$$\sum_{\substack{n,m,i,j\in\mathbb{Z}}}q^{n^2+m^2+i^2+j^2+ij+nm}=a^2(q).$$ In particular, from [6, pg.37, eq.(2.1)] and [6, pg.37, Proposition 2.2] we have the known identity
\begin{equation} a(q)=9q\frac{(q^9;q^9)_{\infty}^3}{(q^3;q^3)_{\infty}}+\frac{(q;q)_{\infty}^3}{(q^3;q^3)_{\infty}}.\end{equation}

Hence,
$$ \sum_{n\ge1}spt_{2,3}(n)q^n-3\sum_{n\ge1}spt(n)q^{3n}$$
\begin{equation}=\frac{1}{(q^3;q^3)_{\infty}}\left(\frac{27}{4}q^2\frac{(q^9;q^9)_{\infty}^6}{(q^3;q^3)_{\infty}^2}+\frac{3q}{2}\frac{(q;q)_{\infty}^3(q^9;q^9)_{\infty}^3}{(q^3;q^3)_{\infty}^2}+\frac{1}{12}\frac{(q;q)_{\infty}^6}{(q^3;q^3)_{\infty}^2}-\frac{1}{12}\right)+\sum_{n\ge1}N_2(n)q^{3n} \end{equation}

\begin{align}\equiv \frac{1}{12}\frac{(q;q)_{\infty}^6}{(q^3;q^3)_{\infty}^3}-\frac{1}{12(q^3;q^3)_{\infty}}+\sum_{n\ge1}N_2(n)q^{3n} \pmod{3}
\end{align}
Therefore,
\begin{equation} \sum_{n\ge1}spt_{2,3}(n)q^n\equiv \frac{1}{12}\frac{(q;q)_{\infty}^6}{(q^3;q^3)_{\infty}^3}-\frac{1}{12(q^3;q^3)_{\infty}}+\sum_{n\ge1}N_2(n)q^{3n} \pmod{3}.\end{equation}
On the other hand, we may write
\begin{equation}\frac{1}{12}\frac{(q;q)_{\infty}^6}{(q^3;q^3)_{\infty}^3}=\frac{1}{12}\left(\sum_{n\ge0}(-1)^n(2n+1)q^{n(n+1)/2}\right)^2\left(\sum_{n\ge0}p(n)q^{3n}\right)^3.\end{equation}
Since we may write 
$$\sum_{n\ge0}f_nq^{3n}=\left(\sum_{n\ge0}p(n)q^{3n}\right)^3,$$ where $f_n$ is a convolution sum involving $p(n),$ we are concerned primarily when the sum of two triangular numbers is $\equiv2\pmod{3}.$ If we write $T_i=i(i+1)/2,$ $i,j\in\mathbb{N},$ then we have that $T_i+T_j\equiv2\pmod{3}$ only when both $i\equiv1\pmod{3},$ $j\equiv1\pmod{3}.$ When this occurs we see that $(2i+1)(2j+1)$ is of the form $9(2i'+1)(2j'+1),$ $i',j'\in\mathbb{N}.$ Therefore, taking the coefficient of $q^{3n+2}$ in (27) now gives us the following result.
\begin{theorem} $spt_{2,3}(3n+2)\equiv 0\pmod{3}.$  \end{theorem}
To see some examples numerically (recall the restriction on parts that parts are $<$ twice the smallest or multiples of three $\ge$ thrice the smallest), we have the following examples:
\newline
{\bf Example 1:} $spt_{2,3}(5)=9\equiv 0\pmod{3}.$ Since we are to count the number of appearances of the smallest parts in the partitions $(5),$ $(3,2),$ $(3,1,1),$ $(1,1,1,1,1).$
\newline
{\bf Example 2:} $spt_{2,3}(8)=27\equiv 0\pmod{3}.$ Since we are to count the number of appearances of the smallest parts in the partitions $(8),$ $(6,2),$ $(6,1,1),$ $(5,3),$ $(4,4),$ $(3,3,2),$ $(3,3,1,1),$ $(3,1,1,1,1,1),$ $(2,2,2,2),$ $(1,1,1,1,1,1,1,1).$

1390 Bumps River Rd. \\*
Centerville, MA
02632 \\*
USA \\*
E-mail: alexpatk@hotmail.com
\end{document}